\documentclass[10 pt]{article}
\usepackage[margin=1in]{geometry}
\newcommand\norm[1]{\left\lVert#1\right\rVert}
\usepackage{geometry}
\usepackage{graphicx}
\usepackage{tikz}
\usetikzlibrary{shapes,shadows,arrows}
\usepackage{amsmath,amsfonts,amssymb,amscd,amsthm,xspace}
\usepackage{amsmath, color}
\theoremstyle{plain}

\newtheorem{theorem}{Theorem}[subsection]

\theoremstyle{definition}

\theoremstyle{remark}

\begin{document}
\title{\itshape Direct and Inverse Source Problem for a Space Fractional Advection Dispersion Equation }
\author{{\small Abeer Aldoghaither,  Taous-Meriem Laleg-Kirati and Da-Yan Liu }\\  {\small E-mail: abeer.aldoghaiher@kaust.edu.sa, taousmeriem.laleg@kaust.edu.sa and dayan.liu@kaust.edu.sa } \\ {\small Mathematical and Computational Sciences and Engineering Division, King
Abdullah University of Science} \\ {\small and Technology (KAUST), Thuwal, Kingdom of
Saudi Arabia}}
\date{}

\maketitle
%---------------------------------------------------------------------------------------------------------------------------------------------	

\begin{abstract}
In this paper, direct and inverse problems for a space fractional advection dispersion equation on a finite domain are studied. The inverse problem consists in determining the source term from a final observation.  We first drive the fundamental solution to the direct problem and we  show that the relation between source and the final observation  is  linear. Moreover, we study the well-posedness of both problems: existence, uniqueness and stability. Finally, we illustrate the  results with a numerical example.

\subsubsection*{keywords} space fractional advection dispersion equation; inverse source problem; ill-posedness, Tikhonov 
%\end{keywords}
\end{abstract}

\section{Introduction}
\vspace{-0.4cm}
Fractional calculus has been used in many applications  in various fields of sciences and engineering, such as  physics, chemistry, biology, Control, electrical and mechanical  engineering, signal processing,  and  finance \cite{podl, Meerschaert, Xiong}.
For instance, when describing anomalous diffusion, such as contaminants transport in the soil, oil flow in porous media,  groundwater flow and turbulence \cite{Wei, Schumer1,Xiong}, there is an evidence that  fractional  models are efficient to capture some important features of particles transport, such as particles with velocity variation and long-rest periods \cite{Schumer2}. This is due to  the  nonlocal property of the fractional operator.   

While direct problem in fractional diffusion equation  is widely considered, work on inverse problem is more recent.  Mainardi \cite{MainardiF}, solved a fractional diffusion wave equation in a one dimensional bounded domain by applying the Laplace transform. Sakamoto  \textit{et al} \cite{SakamotoI}, considered a fractional diffusion equation where they established the unique existence of the weak solution based on eigenfuction expansion. Moreover, they studied the stability and the uniqueness for the backward problem and for the inverse source problem. Pedas \textit{et al} \cite{PedasN}, presented a numerical solution by a piecewise polynomial collocation method. Brunner \textit{et al} \cite{BrunnerN}, presented an algorithm based on an adaptive time stepping and adaptive spatial basis selection approach, to solve a 2D fractional subdiffusion problem.  Wei \textit{et al} \cite{Wei}, considered an inverse source problem for a  space fractional diffusion equation, where they numerically solved the inverse problem using the best perturbation method based on the Tikhonov regularization. Bondarenko \textit{et al} \cite{BondarenkoG,BondarenkoN}, obtained an exact analytic solution of a time fractional diffusion equation where they determined the diffusion coefficient and the derivative order. They also presented a conditionally stable weighted difference scheme and give numerical results by solving a minimization problem with the Levenberg-Marquardt algorithm.

%Moreover,  Kirane \textit{et al} \cite{MalikO} proved the existence and the unique determination of  the source term of a time fractional diffusion equation.In addition, Jin \textit{et al} \cite{Jin} were interested  in proving the unique determination  of  recovering the potential term in a time fractional diffusion equation from the flux measurements and presented a Newton type method to solve the problem numerically. Based on a modification of the exact kernel, Qian \cite{QianO} presented a regularization method to  determine the boundary temperature for a fractional diffusion problem.
 % In ground water contaminant transport, an estimation of the concentration of the source can provide information on its structure, whether it is a heavy metal, an organic substance, or a specific material hazardous the environment  \cite{AndrleI}. Usually we try to estimate the unknown information from available measurements, this is an inverse problem. 

Unlike the fractional diffusion equation which attracted many researchers, work on fractional advection dispersion equation is less-well covered. Salim \textit{et al} \cite{SalimA}, presented an analytic solution for the time fractional advection dispersion equation with a reaction term, by applying the Fourier and Laplace transform. Zheng \textit{et al} \cite{Zheng}, presented a spectral regularization method, based on the solution given by the Fourier method  for the Cauchy problem for a time fractional advection dispersion equation.
Chen \textit{et al} \cite{ChenAD}, considered a two-dimensional fractional advection dispersion equation on a finite domain. They presented an unconditionally stable second order difference method based on the Directions Implicit-Euler method (ADI-Euler) and the unshifted Grunwald formula.  
Liu \textit{et al} \cite{LiuN},  presented a numerical method for a space fractional advection dispersion equation,  where the partial differential equation was transferred into an ODE by applying the method of lines.
Chi \textit{et al}  \cite{Chi}, considered an inverse problem for a space fractional advection dispersion equation where they determined the space-dependent source magnitude from a final observation. They have  solved the inverse problem numerically in presence and in absence of noises using an optimal perturbation regularization algorithm.  However, the stability of the proposed  method depends on the initial guess and the choice of some base functions. Zhang \textit{et al} \cite{Zhang}, have solved this inverse problem using the same optimal perturbation regularization algorithm when the fractional order, the diffusion coefficient and the average velocity are unknown. \\
Huang  \textit{et al} \cite{Huang}, derived the fundamental solution for the direct problem for the time and space fractional dispersion equations in terms of the Green function. However, they only considered  homogenous equations in their study.
To the authors' best knowledge, except \cite{Huang}, there is no other work on the mathematical analysis of the space fractional advection dispersion equation. 
% \cite{Huang,ChenAD,LiuN,Chi, Zhang}
%\textcolor{red}{For a fractional advection dispersion equation, the direct problem consists in finding the concentration c when all of the parameters are known while, the inverse problem consists in finding some unknowns from available measurements. The unknowns can be the velocity, dispersion coefficient, derivative order, or the source term.}

In this paper, we are interested in mathematical analysis of a space fractional advection dispersion equation that can be used in modeling underground water transport in heterogeneous porous media \cite{Schumer1}.   The space fractional advection dispersion equation is in fact used  to   model particles with velocity variation, while time fractional advection dispersion equation is used to model particles with long-rest period \cite{Schumer1}. In addition to the physical importance of such a model, inverse problems play an important role, where unknown physical  quantities  are estimated from available measurements. For instance, in ground water contaminant transport an estimation of the concentration of the source can provide information on its structure, whether it is a heavy metal, an organic substance, or a specific material hazardous the environment.  

The goal of this paper is to mathematically analyze direct and inverse source problems for a space fractional advection dispersion equation. The inverse problem consists in recovering the source term using final observations by assuming that the source is time independent. This paper is organized as follows. In section 2, the problem and some primarily definitions will be introduced.  In section 3, starting from the results obtained in  \cite{Huang} on the solution of a homogeneous fractional dispersion equation,    an exact analytic solution for the non-homogeneous case will be  drived. However, to  overcome some obstacles,  superposition and  Duhamel's principles are used.  Moreover,   existence and uniqueness of the solution will be  established. Section 4 will show that the relation between the unknown source and the  final observation is in fact linear  which simplifies the mathematical and numerical analysis of the inverse problem.  We  will also study  the well-posedness of the inverse problem in term of existence, uniqueness and stability of the solution \cite{Kirsch}. A numerical method based on the Tikhonov regularization will be  presented in section 5 and in section 6,  a numerical example  will be  given. Finally, a conclusion will summarize the obtained results.
%------------------------------
\section{Problem Statement}
We consider the following fractional advection dispersion equation 
\begin{equation}\label{FADE} 
		\left\{
			\begin{array}{ l ll}  	
					\dfrac{\partial c(x,t)}{\partial t} = -\nu \dfrac{\partial c(x,t)}{\partial x} + d \dfrac{\partial ^ \alpha c(x,t) }{\partial x^\alpha} + r(x), &\quad \quad  0<x < L, & \quad t >0 , \\		
					c(x,0)=g_0(x),    \\
     					c(0,t)=0, \\
     					c(L,t)=0, \\
			\end{array}
		\right.
\end{equation}
where $c$ is the concentration, $\nu$ is the average velocity, $d$ the is dispersion coefficient, $r$ is the source term, and $\alpha$ is the fractional order for the space derivative with $1 < \alpha \le 2 $. We assume that $\nu$ and $d$ are constant and the source term depends only on $x$. \\
By assuming that we have both left-to-right and right-to-left flows, we consider the Riesz-Feller fractional derivative of order $\alpha$, defined as follows \cite{Huang}:
	\begin{equation}\label{Riesz_feller}
\begin{split}
		D^\alpha_\theta c(x,t) = \frac{\Gamma(1+\alpha)}{\pi} 				\left \{  \sin[\frac{(\alpha+\theta)\pi}{2}]	 \int_{- \infty}^{+\infty}	  \frac{c(x+\xi,t)-c(x,t)}{\xi ^{1+\alpha}} {\rm d} \xi 	\right.									 \\
 \left. + \sin[\frac{(\alpha-\theta)\pi}{2}]	 \int_{- \infty}^{+\infty}	  \frac{c(x-\xi,t)-c(x,t)}{\xi ^{1+ \alpha}} {\rm d} \xi
 				\right \},
\end{split}
\end{equation}
where $\theta$ is the skewness with  $|\theta| \le \min \{ \alpha, 2-\alpha \}$. The Riesz-Feller fractional derivative can also be defined by \cite{ZhengT}:
	\begin{equation}\label{Riesz_feller}	
\begin{split}
		D^\alpha_\theta c(x,t) = -\frac{1}{\Gamma(2-\alpha) \sin(\alpha \pi)}			\left \{  \sin \frac{(\alpha-\theta)\pi}{2}	 \frac{{\rm d}^2}{{\rm d} x^2} \int_{-\infty}^{x}	  \frac{c(\xi,t)}{(x - \xi) ^{\alpha-1}} {\rm d} \xi 	\right.									 \\
 \left. + \sin \frac{(\alpha+\theta)\pi}{2}	\frac{{\rm d}^2}{{\rm d} x^2} \int_x^{+\infty}	  \frac{c(\xi,t)}{(\xi-x) ^{1+ \alpha}} {\rm d} \xi
 				\right \}.
\end{split}
\end{equation}
The Fourier transform of the Riesz-Feller fractional derivative is defined by \cite{Huang}:
				\begin{equation}\label{FT_RF}	
						\mathcal{F}\{ D^\alpha _\theta c(x,t)\} = -\psi_\theta ^\alpha(k) \hat c(k,t),	
				\end{equation}
with
	\begin{equation}\label{psi}
 			\psi_\theta ^\alpha(k)= |k|^\alpha e^{{\rm i}(sign (k))\theta \pi/2},
	\end{equation}
where $\hat c$ denotes the Fourier transform of $c$, and $k$ is the variable in the frequency domain.

We consider the inverse problem that consists in finding the source term $r$ where the concentration $c$ is also unknown except at a final time $t=T$:
\begin{eqnarray}\label{prior1}
			\begin{array}{ll}
				c(x,T)=g_T(x), &  0 < x < L,
			\end{array}
	\end{eqnarray}
where $g_T$ is the measured concentration.  We assume that the velocity $\nu$, the dispersion coefficient $d$, and the derivative order $\alpha$ are known.

In this paper, we first solve the direct problem and present an exact analytic solution to (\ref{FADE}).  We show that the operator relating the unkown source to the final observation is linear which simplifies the analysis of the properties of the inverse problem and its numerical solution. 
\section{Direct Problem}
\vspace{-0.4cm}
In this section, we construct a closed-form analytic solution for the direct problem of (\ref{FADE}), which gives a linear operator which  simplifies the mathematical analysis of existence and uniqueness of the solution.  
\subsection{Analytic Solution }
%To find the solution of the direct problems,
 We propose to use the Fourier transform method, which is commonly used to solve fractional differential equations and transfer  non-linear equation to linear equation \cite{ZhengT,QianO}. Then, we use the integrating factor to obtain the solution of the direct problem. In  order to apply the Fourier transform to (\ref{FADE}), we propose to  extend  $c$ and $r$ to the whole real line by defining: $\forall x \in \mathbb{R}, \, \forall t \in \mathbb{R}^*_+$,
\begin{equation}
u(x,t)=c(x,t) \cdot p(x),
\end{equation}
\begin{equation}
f(x) = r(x)\cdot p(x),
\end{equation}
where
\begin{equation}
p(x)= \left \{
		\begin{array}{ll}
		1, & \text{ if } \, \, x \in ]0,L[, \\
		 0, & \text{  else}.
		\end{array}	\right.
\end{equation}
Then, we propose to solve the following equation:  
\begin{equation}\label{u_FADE}
		\left\{
			\begin{array}{ l ll}  	
					\dfrac{\partial u(x,t)}{\partial t} = -\nu \dfrac{\partial u(x,t)}{\partial x} + d \dfrac{\partial ^ \alpha u(x,t) }{\partial x^\alpha} + f(x), & \quad \quad  x \in \mathbb{R}, &\quad  t >0 , \\		
					u(x,0)=g_0(x).   \\
			\end{array}
		\right.
\end{equation}
\begin{theorem}
Assuming that $f ,g_0\in L^{2}(\mathbb{R})$, then there exists a solution of the system (\ref{u_FADE}), which can be given as follows:
	\begin{equation}\label{Solution_DP}
		u(x,t)=   \int_{- \infty}^{+ \infty} \left \{ \int_{0}^{t}   G_\alpha^\theta(x-y,t-\tau) {\rm d} \tau \right \} f(y)  {\rm d} y +   \int_{- \infty}^{+ \infty}    G_\alpha^\theta(x-y,t) g_0(y)   {\rm d} y,
	\end{equation}
	where \begin{equation}\label{Green}
	G_\alpha^\theta(x,t)= \frac{1}{2\pi} \int_{- \infty}^{+ \infty} {\rm e}^{-{\rm i} kx} \hat G_\alpha^\theta(k,t){\rm d} k,
\end{equation}
with
\begin{equation}\label{F_Green}
\hat G_\alpha^\theta (k,t)= {\rm e}^{({\rm i} \nu k- d \psi_\theta^\alpha(k))t }.%= {\rm e}^{{\rm i} \nu k t} {\rm e}^{- d \psi_\theta^\alpha(k)t }.
\end{equation}

\end{theorem}
We  refer to Appendix A for more details about the green function $G_\alpha^\theta$.

\begin{proof}
Applying the superposition principle, the solution of equation (\ref{u_FADE}) can be written as $u = v + w$, where $w$ and $v$ are solutions to the following equations, respectively. 
\begin{equation}\label{w}
		\left\{
			\begin{array}{ l ll}  	
     				\dfrac{\partial w(x,t)}{\partial t} = -\nu \dfrac{\partial w(x,t)}{\partial x} + d \dfrac{\partial^\alpha w(x,t)}{\partial x^\alpha}, & \quad \quad x \in \mathbb{R}, & \quad  t>0, \\
				w(x,0)=g_0 (x),  \\
			\end{array}
		\right.
\end{equation}	
\begin{equation}\label{v}
		\left\{
			\begin{array}{ l ll}  	
     				\dfrac{\partial v(x,t)}{\partial t} = -\nu \dfrac{\partial v(x,t)}{\partial x} + d \dfrac{\partial^\alpha v(x,t)}{\partial x^\alpha} + f(x), & \quad \quad x \in \mathbb{R}, &\quad   t>0, \\
				v(x,0)=0.  \\
			\end{array}
		\right.
\end{equation}

Applying the Duhamel's principle \cite{Jeffrey}, the solution $v$ of (\ref{v}) is given by
 \begin{equation}\label{D_principle}
 v(x,t)= \int_0^t V(x,t,\tau) {\rm d} \tau,
 \end{equation} where $\tau$ is fixed with $\tau \in ]0,t[$ and $V(\cdot,\cdot:\tau)$ is the solution of the following equation:
  \begin{equation}\label{V}
		\left\{
			\begin{array}{ l ll}  	
     				\dfrac{\partial V(x,t;\tau)}{\partial t} = -\nu \dfrac{\partial V(x,t;\tau)}{\partial x} + d \dfrac{\partial^\alpha V(x,t;\tau)}{\partial x^\alpha}, & \quad \quad  x \in \mathbb{R}, & \quad  t>0, \\
				V(x;t=\tau)= f(x).\\
			\end{array}
		\right.
\end{equation}

Therefore, we start solving (\ref{V}) by applying the Fourier transform. Then we get:
\begin{equation}
	\mathcal{F} \{\frac{\partial V(x,t)}{\partial t}\} = -\nu \mathcal{F}\{ \frac{\partial V(x,t)}{\partial x}\}+ d \mathcal{F} \{\frac{\partial ^ \alpha V(x,t) }{\partial x^\alpha}\},
\end{equation}
\begin{equation}
\frac{\partial \hat{V}(k,t)}{\partial t} = -\nu (-{\rm i} k)^1  \hat{V}(k,t) - d \psi^\alpha_\theta (k) \hat{V}(k,t),
\end{equation}
where the Fourier transform $\hat V$ of $V$ is obtained using formulas (\ref{FT_RF}) and (\ref{psi}).
\\ Then, we have:
\begin{equation}\label{eq_1}
	\frac{\partial \hat{V}(k,t)}{\partial t} = \left [ \nu {\rm i} k - d \psi^\alpha_\theta (k) \right ]\hat{V}(k,t).
\end{equation}
By multiplying (\ref{eq_1}) by the integrating factor  $exp[\int - [ \nu {\rm i} k - d \psi^\alpha_\theta (k)] {\rm d}t] = exp \{ [- \nu {\rm i} k + d \psi^\alpha_\theta (k)]t\} $, we get:
\begin{equation}
	\frac{\partial \hat{V}(k,t)}{\partial t} {\rm e}^{  [- \nu {\rm i} k + d \psi^\alpha_\theta (k)]t} - {\rm e}^{  [- \nu {\rm i} k + d \psi^\alpha_\theta (k)]t}  [ \nu k {\rm i} -  d \psi^\alpha_\theta(k)] \hat{V}(k,t) =0,
\end{equation}
\begin{equation}
	\frac{{\rm d}}{{\rm d}t} \left ( \hat{V}(k,t)  {\rm e}^{  [- \nu {\rm i} k + d \psi^\alpha_\theta (k)]t} \right ) = 0.
\end{equation}
Then, by integrating with respect to $t$, it yields:
\begin{equation}\label{constant}
	\hat{V}(k,t) = {\rm e}^{  [ \nu {\rm i} k - d \psi^\alpha_\theta (k)]t}  c_1,
\end{equation}
where $c_1 \in \mathbb{R}$ is a constant.
Applying the initial condition $\hat{V}(k,t=\tau)=\hat{f}(k)$ to (\ref{constant}) gives us:
\begin{equation}
	\hat f(k) ={\rm e}^{  [ \nu {\rm i} k - d \psi^\alpha_\theta (k)]\tau}  c_1,
\end{equation}
\begin{equation}\label{c_1}
	c_1= {\rm e}^{ - [ \nu {\rm i} k - d \psi^\alpha_\theta (k)]\tau} \hat f(k).
\end{equation}
Substituting (\ref{c_1}) into (\ref{constant}) gives us the solution of (\ref{V}):
\begin{equation}\label{V_fourier}
	\hat{V}(k,t) = {\rm e}^{  [ \nu {\rm i} k - d \psi^\alpha_\theta (k)](t-\tau)} \hat f(k).
\end{equation}
Applying the inverse Fourier Transformation to (\ref{V_fourier}), we obtain:
\begin{equation}
	V(x,t) = \frac{1}{2\pi}\int_{-\infty}^{\infty} {\rm e}^{  [ \nu {\rm i} k - d \psi^\alpha_\theta (k)](t-\tau)} \hat f(k) {\rm e}^{-{\rm i} kx} {\rm d} k.
\end{equation}
Finally, using (\ref{D_principle}) gives us the solution $v$ of (\ref{v}):
\begin{equation}
v(x,t)= \frac{1}{2\pi} \int_{0}^{t} \int_{-\infty}^{\infty} {\rm e}^{  [ \nu {\rm i} k - d \psi^\alpha_\theta (k)](t-\tau)} \hat f(k) {\rm e}^{-{ \rm i} kx} {\rm d} k {\rm d} \tau,
\end{equation}
% which can be written in terms of the green's function given in (\ref{F_Green}), as:
which can be written as:
\begin{equation}\label{solution_v}
v(x,t)= \frac{1}{2\pi} \int_{0}^{t} \int_{-\infty}^{\infty}  \hat G_\alpha^\theta (k,t-\tau) \hat f(k) {\rm e}^{-{\rm i} kx} {\rm d} k {\rm d} \tau.
\end{equation}
 Using the same technique the solution $w$ of (\ref{w}) can be obtained:
		\begin{equation}\label{solution_w}
			w(x,t)=	\frac{1}{2\pi}\int_{- \infty}^{+ \infty} \hat G_\alpha^\theta(k,t) \hat g_0(k) {\rm e}^{-{\rm i} kx} {\rm d} k,
		\end{equation}
where
	$\hat{g}_0$ is the Fourier transform of $g_0$. 
	
Then, by adding (\ref{solution_v}) and (\ref{solution_w}), we can get the solution of (\ref{u_FADE}):
	\begin{equation}\label{DP_solution}
		u(x,t)=  \frac{1}{2\pi}\int_{0}^{t}  \int_{- \infty}^{+ \infty}  \hat G_\alpha^\theta(k,t-\tau) \hat f(k) {\rm e}^{-{\rm i} kx} {\rm d} k {\rm d} \tau +   \frac{1}{2\pi}\int_{- \infty}^{+ \infty} \hat G_\alpha^\theta(k,t) \hat g_0(k) {\rm e}^{-{\rm i} kx} {\rm d} k.
	\end{equation}
By inverting the Fourier transform of $\hat f$ and $\hat g_0$ in (\ref{DP_solution}), we get:
\begin{equation}
	\begin{split}
	u(x,t)=  & \frac{1}{2\pi}\int_{0}^{t}  \int_{- \infty}^{+ \infty} {\rm e}^{-{\rm i} kx} \hat G_\alpha^\theta(k,t-\tau)   \int_{- \infty}^{+ \infty} f(y) {\rm e}^{{\rm i} ky} {\rm d} y {\rm d} k  {\rm d} \tau  \\ & +
	 \frac{1}{2\pi}\int_{- \infty}^{+ \infty} {\rm e}^{-{\rm i} kx} \hat G_\alpha^\theta(k,t)  \int_{- \infty}^{+ \infty} g_0(y) {\rm e}^{{\rm i} ky}  {\rm d} y {\rm d} k,
	\end{split}
\end{equation}
\begin{equation}
\begin{split}
	=  & \int_{- \infty}^{+ \infty} \left \{ \int_{0}^{t} \left [ \frac{1}{2\pi}  \int_{- \infty}^{+ \infty} {\rm e}^{-{\rm i} k(x-y)} \hat G_\alpha^\theta(k,t-\tau)  {\rm d} k \right ] {\rm d} \tau \right \} f(y)  {\rm d} y   \\
	&  +  \int_{- \infty}^{+ \infty} \left [\frac{1}{2\pi} \int_{- \infty}^{+ \infty} {\rm e}^{-{\rm i} k(x-y)} \hat G_\alpha^\theta(k,t) {\rm d} k \right ]   g_0(y)   {\rm d} y.
\end{split}
\end{equation}
Then, the fundamental solution of (\ref{u_FADE}) is:
\begin{equation}\label{Solution_DP}
	u(x,t)=   \int_{- \infty}^{+ \infty} \left \{ \int_{0}^{t}   G_\alpha^\theta(x-y,t-\tau) {\rm d} \tau \right \} f(y)  {\rm d} y +   \int_{- \infty}^{+ \infty}    G_\alpha^\theta(x-y,t) g_0(y)   {\rm d} y.
\end{equation}
\end{proof}

\subsection{Uniqueness}
The uniqueness of the solution is a direct result from the linearity of the integral operator.
\begin{theorem}
Let $u_1$ and $u_2$ be two solutions of (\ref{u_FADE}) with sources $f_1$ and $f_2$, respectively. Then, the condition $f_1=f_2$ implies that $u_1=u_2$.
\end{theorem}
\begin{proof}
Suppose that $f_1-f_2=0$. Since $u_1$ and $u_2$ are solutions of (\ref{u_FADE}), we have:
\begin{equation}
	u_1(x,t)=   \int_{- \infty}^{+ \infty} \left \{ \int_{0}^{t}   G_\alpha^\theta(x-y,t-\tau) d\tau \right \} f_1(y)  {\rm d} y +   \int_{- \infty}^{+ \infty}    G_\alpha^\theta(x-y,t) g_0(y)  { \rm d} y,
\end{equation}
\begin{equation}
	u_2(x,t)=   \int_{- \infty}^{+ \infty} \left \{ \int_{0}^{t}   G_\alpha^\theta(x-y,t-\tau) {\rm d} \tau \right \} f_2(y)  {\rm d} y  +    \int_{- \infty}^{+ \infty}    G_\alpha^\theta(x-y,t) g_0(y)   {\rm d} y.
\end{equation}
Then the following equality
\begin{equation}
	u_1(x,t) - u_2(x,t)=   \int_{- \infty}^{+ \infty} \left \{ \int_{0}^{t}   G_\alpha^\theta(x-y,t-\tau) {\rm d} \tau \right \} [f_1(y) - f_2(y)] {\rm d} y =0
\end{equation}
completes the proof.
\end{proof}

\section{Inverse Problem}
\vspace{-0.4cm}
In this section, we study some properties of the inverse source problem (ISP) consisting in the estimation of the  source from the knowledge of the concentration at  time $t=T$:
$$
f = K(g_T) \quad \quad \quad (ISP)
$$
  We will study the ill-posedness of the problem in the sense of Hadamard \cite{Kirsch}. 
\subsection{Existence }
	The following theorem shows that the measurement can determine the source term.
		\begin{theorem}
	Assuming that $f, g_0 \in L^2(\mathbb{R})$, then the solution of the inverse problem for the system (\ref{u_FADE}), which determines  the source term $f$ using the measurement given in (\ref{prior1}), exists.
	\end{theorem}
	\begin{proof}
	 Substituting (\ref{prior1}) into (\ref{Solution_DP}), we get:
	\begin{equation}
		g_T(x)=   \int_{- \infty}^{+ \infty} \left \{ \int_{0}^{T}   G_\alpha^\theta(x-y,T-\tau) d\tau \right \} f(y)  {\rm d} y +   \int_{- \infty}^{+ \infty}    G_\alpha^\theta(x-y,T) g_0(y)   {\rm d} y.
	\end{equation}
	Then $f$ is the solution of the following equation:
	\begin{equation}\label{conv}
		 \int_{- \infty}^{+ \infty} \left \{ \int_{0}^{T}   G_\alpha^\theta(x-y,T-\tau) d\tau \right \} f(y)  {\rm d} y = h(x),
	\end{equation}
	where
	\begin{equation}
	h(x):=g_T(x) -  \int_{- \infty}^{+ \infty}    G_\alpha^\theta(x-y,T) g_0(y) {\rm d} y.
	\end{equation}

Since (\ref{conv}) is a convolution, the solution exists.

%\textcolor{red}{some conditions on the smoothness and decay of $G, g_T$ and $g_0$ have to be stated and verified}
\end{proof}

\subsection{Uniqueness}
In this section, we prove the unique determination of the source term from a given measurements at a specific or a final time.

\begin{theorem}\label{T_2}
Let $u_1(\cdot,T)$ and $u_2(\cdot,T)$ be solutions of (\ref{u_FADE}) with sources $f_1$ and $f_2$, respectively. Then, the condition $u_1(\cdot,T)=u_2(\cdot,T)$ implies that $f_1=f_2$ almost everywhere.
\end{theorem}
\begin{proof}
Let $u_1(x,T)=u_2(x,T)$.  Since $u_1$ and $u_2$ are solutions of (\ref{u_FADE}), then using (\ref{DP_solution}) we get:
	\begin{equation}
		u_1(x,T) - u_2(x,T)  = \frac{1}{2\pi}\int_{0}^{T}  \int_{- \infty}^{+ \infty}  \hat G_\alpha^\theta(k,T-\tau)  [\hat f_1(k) -\hat f_2(k)]  {\rm e}^{-{\rm i}kx} {\rm d} k \, {\rm d} \tau,
	\end{equation}
which is equivalent to:	
		\begin{equation}
		0  = \frac{1}{2\pi}\int_{0}^{T}  \int_{- \infty}^{+ \infty}  \hat G_\alpha^\theta(k,T-\tau)  [\hat f_1(k) -\hat f_2(k)]  {\rm e}^{-{\rm i} kx} {\rm d} k \, {\rm d} \tau.
	\end{equation}
By applying the Fourier transform, we get:
	\begin{equation}\label{uni}
		 \int_{0}^{T}   \hat G_\alpha^\theta(k,T-\tau) {\rm d} \tau\,  [\hat f_1(k) -\hat f_2(k)]    = 0.
	\end{equation}
Computing the following integration gives us:
\begin{equation}
		 \int_{0}^{T}  \hat G_\alpha^\theta(k,T-\tau) {\rm d}\tau =    \frac{1 }{(i\nu k- d \psi_\theta^\alpha(k))} \left [  {\rm e}^{(i\nu k- d \psi_\theta^\alpha(k))T }- 1	\right ].
\end{equation}
Therefore, equation (\ref{uni}) holds if and only if
	\begin{equation}
	  \frac{1 }{({\rm i}\nu k- d \psi_\theta^\alpha(k))} \left [  {\rm e}^{({\rm i}\nu k- d \psi_\theta^\alpha(k))T }- 1	\right ] = 0 \Longleftrightarrow  	({\rm i}\nu k- d \psi_\theta^\alpha(k))T = 0,
	\end{equation}
which implies that:
	%\begin{eqnarray}\label{ri} - d |k|^\alpha \cos (\theta \pi/2)+ i(\nu k \pm d |k|^\alpha \sin \theta \pi /2) = 0.  \end{eqnarray}
	 \begin{equation}
		\begin{array}{ccc}
			Re({\rm i}\nu k- d \psi_\theta^\alpha(k))=0, & \:\:\:\mbox{and}\:\:\: &Im( {\rm i}\nu k- d \psi_\theta^\alpha(k))=0.
		\end{array}
	\end{equation}
The real part  
\begin{equation}
	Re({\rm i}\nu k- d \psi_\theta^\alpha(k))=0,
\end{equation}
if
	 \begin{equation}
	\nu k \pm d |k|^\alpha \sin \theta \pi /2= 0.
	\end{equation}
	Then either $k=0$, or $\theta=1$, but $\theta \neq 1$, since $\theta \le 2-\alpha$ and $\alpha >1$.
	%\begin{eqnarray}\fl
	%Re(- d |k|^\alpha \cos (\theta \pi/2)+ i(\nu k \pm d |k|^\alpha \sin \theta \pi /2))=	- d |k|^\alpha \cos (\theta \pi/2) = 0,
	%\end{eqnarray} 
	\\
The imaginary part 
	\begin{equation}
		 Im({\rm i}\nu k- d \psi_\theta^\alpha(k))=0,
	\end{equation}
	if 
		\begin{equation}
		\nu k \pm d |k|^\alpha \sin \theta \pi /2= 0,
	\end{equation}
	 \begin{equation}\Rightarrow \left\{
												\begin{array}{ l }  	
     													k=0   \\
     													\eta(k):= \frac{|k|^\alpha}{k} = \mp \frac{\nu}{d \sin(\theta \frac{\pi}{2})}   \\ 
     		
												\end{array} 
											\right.    .
	\end{equation}
	This implies that: 
\begin{equation}
\begin{array}{ll}
 \int_{0}^{T}  \hat G_\alpha^\theta(k,T-\tau) {\rm d}\tau \neq 0, & \quad \quad \forall k \in \mathbb{R}.
\end{array}
\end{equation}
Therefore
\begin{equation}
 \hat f_1(k) - \hat f_2(k)=0,
\end{equation}
which completes the proof.
\end{proof}\subsection{Stability}
In this subsection, we show that the third Hadamard condition \cite{Kirsch}, is not satisfied, $ i.e.$ an arbitrarily small error in the measurement data lead to a large error in the solution \cite{Kirsch}. To prove this unstability, we show that a bounded  perturbation in the source will not affect the final observation of the concentration \cite{Kirsch}. 
\begin{theorem}
The inverse problem ISP is not stable in the sense of Hadamard. 

\end{theorem}
\begin{proof}

We asume that  $u$ and $u^\delta$ are solutions of (\ref{u_FADE}) with sources $f$ and $f^\delta$ respectively. 
We suppose that there exists $\delta_n = f_{\delta_n} - f $ with $\lim_{n\rightarrow \infty}{\parallel \delta_n\parallel} =\delta$, $\delta\neq 0$ and  we show that  $ \norm{ u^\delta(\cdot,T) -  u(\cdot,T)}  \rightarrow 0$.

	 We take  $\delta_n(x)=A \sin(\frac{n\pi x}{L})$, where $A$ is a constant, we obtain:
	\begin{eqnarray}
		\norm{f^ \delta (x) - f(x)}_2^2 = \norm{f(x) + A \sin(\frac{n\pi x}{L})- f(x)}_2^2
	\end{eqnarray}
\begin{eqnarray}
	 = \norm{ A \sin(\frac{n\pi x}{L})}_2^2 = \int_0^L A^2 \sin^2(\frac{n\pi x}{L}) {\rm d} x = \frac{1}{2} A^2 L \neq 0.
	\end{eqnarray}
	\begin{equation}
		 \norm{u^\delta(x,T) -  u(x,T)}_2^2 =  \norm{\int_{- \infty}^{+ \infty} \left \{ \int_{0}^{T}   G_\alpha^\theta(x-y,T-\tau) {\rm d} \tau \right \} [ f^\delta(y) - f(y)]  {\rm d} y} _2^2.
	\end{equation}
Then, applying the Riemann-Lebesgue lemma we get \cite{Furdui}:
	\begin{equation}
	\norm{ u^\delta(x,T) -  u(x,T)}_2^2 = \norm{\int_{- \infty}^{+ \infty} \left \{ \int_{0}^{T}   G_\alpha^\theta(x-y,T-\tau) {\rm d} \tau \right \} [  A \sin(\frac{n\pi y}{L}))]  {\rm d} y  }_2^2 \longrightarrow 0.
	\end{equation}
\end{proof}

\section{Numerical Analysis}
\vspace{-0.4cm}
In this section, we present a numerical solution to reconstruct the source term for the space fractional advection dispersion equation (\ref{FADE}) from the measurement $c(\cdot,T)$. 	

\subsection{Finite Difference Method for the Direct Problem}

For the direct problem, we discretize (\ref{FADE}) using a finite difference scheme similar to the one introduced by Meerschaert and  Tadjeran \cite{Meerschaert}. The shifted Grunwald formula is used to discretize the Risze-Feller fractional derivative \cite{ZhengT, ZhangN, MeerschaertF}, as follows:
	\begin{equation}\label{shifted_G}
		D^\alpha _\theta c(x,t) = -[a_r \, _{-\infty}D_x^\alpha c(x,t)+ a_l \, _xD_{\infty}^\alpha c(x,t)],
	\end{equation}
	where
	\begin{equation}
				_{-\infty}D_x^\alpha c(x,t) = \lim_{N \to \infty} \dfrac{1}{h^\alpha}  \sum_{k=0}^N \xi_{\alpha, k} c[x-(k-1)h,t],
	\end{equation}
	\begin{equation}
				_xD_{\infty}^\alpha c(x,t)= \lim_{N \to \infty} \dfrac{1}{h^\alpha}  \sum_{k=0}^{N-i+1} \xi_{\alpha, k}  c[x+(k-1)h,t],
	\end{equation}
	\begin{equation}
 \begin{array}{ll}
 a_r= \dfrac{\sin \frac{(\alpha-\theta)\pi}{2}}{\sin (\alpha \pi)}, &  a_l= \dfrac{\sin \frac{(\alpha+\theta)\pi}{2}}{\sin (\alpha \pi)},
 \end{array}
\end{equation}
and $\xi_{\alpha, k}$ is the  normalized Grunwald weight  defined by:
	\begin{equation}
	\xi_{\alpha,k}=\frac{\Gamma (k-\alpha)}{\Gamma (-\alpha) \Gamma (k+1)}.
 	\end{equation}
	
%	\begin{equation}\label{shifted_G} D^\alpha _\theta c(x,t) = \lim_{N \to +\infty} - \frac{1}{h^\alpha} \left [a_r \sum_{k=1}^N \xi_{\alpha, k} c[x-(k-1)h,t]  + a_l  \sum_{k=0}^{N-i+1} \xi_{\alpha, k}  c[x+(k-1)h,t]  \right ] \end{equation}

	Then, the explicit Euler and the finite difference methods are used to discretize the time and the spatial derivatives, respectively \cite{Meerschaert}:
		\begin{equation}\label{Euler_finite}
			\begin{array}{ccc}
 				\dfrac{\partial c(x,t)}{\partial t}=\dfrac{c_i^{j+1}-c_i^j}{\Delta t }, & &
 				\dfrac{\partial c(x,t)}{\partial x}=\dfrac{c_{i}^{j+1}-c^{j+1}_{i-1}}{\Delta x },
			\end{array}
 		\end{equation}
	where $\Delta t$ and $\Delta x$ are the time step and space step, respectively.

Substituting (\ref{shifted_G}) and (\ref{Euler_finite}) in (\ref{FADE}), we get the following discretization form:
	\begin{equation}\label{dis_2}
		 \frac {c^{j+1}_i -  c^{j}_i}{\Delta t} =  -  \nu \frac{(c^{j+1}_{i} - c^{j+1}_{i-1})}{\Delta x} + d \, \delta_{\alpha, x}  c^{j+1}_i+ r^{j+1}_i,
	\end{equation}
	where
	\begin{equation}
		 \delta_{\alpha,x} c^j_i =- \frac{1}{(\Delta x)^\alpha} [ a_r \sum_{k=0}^{i+1} \xi_{\alpha,k} c^j_{i-k+1} + a_l \sum_{k=0}^{N-i+1} \xi_{\alpha,k} c^j_{i+k-1}] ,
 	\end{equation}
with $ i=1,\dots, N-1 \,  \, \mathrm{and} \, \,  j=1,2,\dots$
with $t_j=j\Delta t$, $x_i=i \Delta x$, $c^j_i=c(x_i,t_j)$, $r_i=r(x_i)$, $c_0^j=0$, and $c_{N}^j=0.$ \\
Then, we get:
	\begin{equation}\label{discrete_form}
 	\begin{array}{lll}
 			(1-d \Delta t \delta _{\alpha, x}) c^{j+1}_i   + \dfrac{\Delta t}{\Delta x} \nu (c^{j+1}_{i} - c_{i-1}^{j+1} )= c_i^j+r_i \Delta t.
	 \end{array}
\end{equation}
Thus, the matrix form of the implicit finite difference scheme (\ref{discrete_form}) is given by:
	 \begin{equation}\label{matrix 1}
	[({\textrm{\bf I}} - {\textrm{\bf  G }} - {\textrm{\bf  L}})+{\textrm{\bf V}}] {\textrm{\bf C}}^{j+1} ={\textrm{\bf C}}^j + {\textrm{\bf R}}
	\end{equation} 
for $n=1,2,\cdots, N-1$ and $m=1,2,\cdots,N-1$, where
\begin{equation}
	{\textrm{\bf  G}}(m,n)=a_r\frac{d (\Delta t) }{(\Delta x)^\alpha}
		\left\{
			\begin{array}{l  l}
      			\xi_{\alpha,m-n+1}    & n\le m-1, \\
     				 \xi_{\alpha,1}    & n= m, \\
			  	 \xi_{\alpha,0}    & n= m+1, \\
				 0 & \text{ else}.\\
			\end{array}
		\right.
\end{equation}
\begin{equation}
	{\textrm{\bf  L}}(m,n)=a_l\frac{d (\Delta t) }{(\Delta x)^\alpha}
		\left\{
			\begin{array}{l  l}
      			\xi_{\alpha,m-n+1}    & n\ge m+1, \\
     				 \xi_{\alpha,1}    & n= m, \\
			  	 \xi_{\alpha,0}    & n= m-1, \\
				 0 & \text{ else}.\\
			\end{array}
		\right.
\end{equation}
\begin{equation}
	{\textrm{\bf  {\textrm{\bf  V}}}}(m, n)= \frac{v \Delta t }{\Delta x}
		\left\{
			\begin{array}{l  l  l}
      				1 & n=m,   \\
      				-1 & m=n-1, \\
     				0 & \text{ else}. \\	
			\end{array}
		\right.
\end{equation}
		\begin{equation} 	
		{\textrm{\bf  R}}= \Delta t \times[ r_1, r_2, \cdots, r_{N-1}]^{\rm T},
	 \end{equation}
 	\begin{equation}
		{\textrm{\bf  C}}^{j+1}=[ c_1^{j+1}, c_2^{j+1}, \cdots, c_{N-1}^{j+1}]^{\rm T}.
	\end{equation}
Consequently, when estimating the source term, the direct problem can be solved using the implicit difference method.

\subsection{Inverse Problem}
  Let ${\textrm{\bf  A}} = [({\textrm{\bf I}} - {\textrm{\bf  G}} - {\textrm{\bf  L}})+{\textrm{\bf  V}}] ^{-1}$, then equation (\ref{matrix 1}) can be written as:
  	\begin{equation}
		{\textrm{\bf  C}}^{j+1}= {\textrm{\bf  A}} ({\textrm{\bf  C}}^{j} + {\textrm{\bf  R}}).
	\end{equation}
By induction, we get the following equation:
 	\begin{equation}\label{equation 2}
		 {\textrm{\bf C}}^N-{\textrm{\bf  A}}^N {\textrm{\bf C}}^0=({\textrm{\bf  I}}-{\textrm{\bf  A}})^{-1} ({\textrm{\bf  I}}-{\textrm{\bf  A}}^N){\textrm{\bf  A}} {\textrm{\bf  R}},
	\end{equation}
which can be written in the following form:
	\begin{equation}
	{\textrm{\bf  Y}}={\textrm{\bf  K}}{\textrm{\bf  R}},
	\end{equation}
where
   	\begin{equation}
	 {\textrm{\bf K}}=({\textrm{\bf I}}-{\textrm{\bf A}})^{-1} ({\textrm{\bf I}}-{\textrm{\bf A}})^N {\textrm{\bf A}},
	\end{equation}
and
	 \begin{equation}
	  {\textrm{\bf Y}}={\textrm{\bf C}}^N-{\textrm{\bf A}}^N {\textrm{\bf C}}^0.
	 \end{equation}
In order to estimate the unknown source, we propose to minimize the following cost function with the
Tikhonov regularization:
\begin{equation}\label{cost_fun}
			J_{\lambda}( {\textrm{\bf R}})= \norm{ {\textrm{\bf Y}} -  {\textrm{\bf KR}}}_2^2  + \lambda \, \Omega( {\textrm{\bf R}}),
		\end{equation}
where $ {\textrm{\bf Y}}$ is the observation and $\Omega( {\textrm{\bf R}})= \norm{\dfrac{{\rm d} ^m {\textrm{\bf R}}}{{\rm d} x^m}}$ with $m=0,1$ is a stabilization functional of $ {\textrm{\bf R}}$ which usually includes a priori information on the problem. The regularization parameter $\lambda$ can be determined using the L-curve \cite{Kirsch}.
\section{Numerical Example}
\vspace{-0.4cm}
Let us consider the following space fractional advection dispersion equation:
	\begin{eqnarray}\label{example_1}
		\begin{array}{lll}
			 \dfrac{\partial c(x,t)}{\partial t} = -0.3 \dfrac{\partial c(x,t)}{\partial x} + 3 \dfrac{\partial ^ {1.5} c(x,t) }{\partial x^{1.5}} + r(x), & 0<x< 7, & t >0,
		\end{array}
 	\end{eqnarray}
 with the following initial and Dirichlet boundary conditions:
\begin{eqnarray}\label{condition_1}
	\left\{
		\begin{array}{ l }  	
     			c(x,0)=0,   \\
     			c(0,t)=0, \\
     			c(7,t)=0. \\
		\end{array}
	\right.
\end{eqnarray}
In this example, we assume that the source term $r(x) =5 \sin \frac{2\pi}{7} x$ is unknown. \\

Figure \ref{DP}, represents the numerical solution of (\ref{example_1}) with the conditions (\ref{condition_1}) at  time $T=1$, which will be used as the exact solution when recovering the source term numerically.
\begin{figure}[http]
	\begin{center}
		\includegraphics[width=3in]{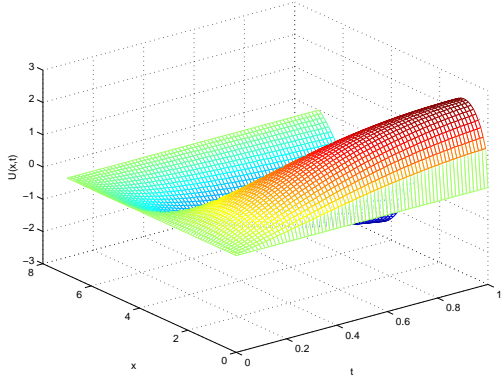}
		 \caption{Numerical solution of the direct problem.}
		 	\label{DP}
	\end{center}
\end{figure}

\begin{figure}[http]
	\begin{center}
		\includegraphics[width=3in]{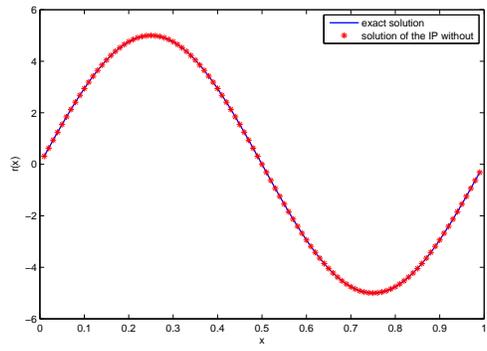}
		 \caption{Solution of the inverse problem without noise.}
		 	\label{Inverse}
	\end{center}
\end{figure}
 \begin{figure}[http]
	\begin{center}
		\includegraphics[width=3in]{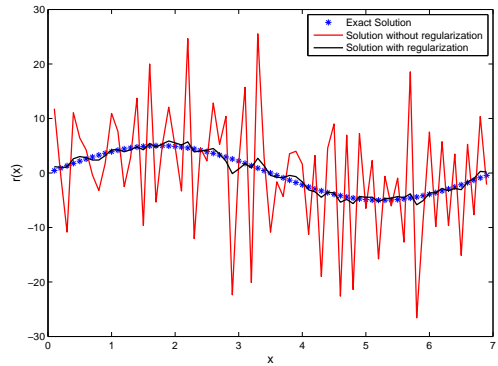}
		 \caption{Solution of the inverse problem with and without regularization with $5\%$ noisy measurements.}
		 	\label{IP}
	\end{center}

\end{figure}

\begin{figure}[http]
	\begin{center}
		\includegraphics[width=3in]{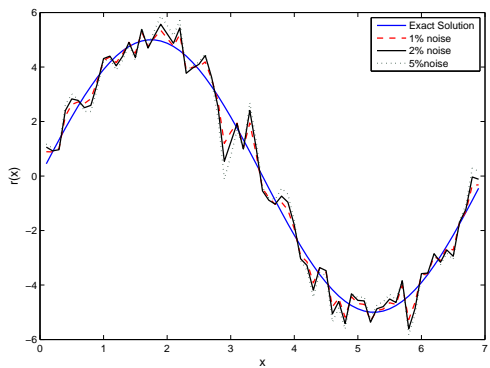}
		 \caption{The exact and the regularized solution with different noise levels.}
		 	\label{noises}
	\end{center}
\end{figure}

Figure \ref{Inverse}, represents the approximated solution $r^\delta$ of the inverse problem without regularization in noise free case. As we can see, if we have the exact measurement  ($i.e.$ without noise) then, the numerical solution matches the approximated solution. While, recovering the source term from noisy measurements is severely ill-posed (see Figure \ref{IP}).

 In Figure \ref{IP}, a comparison of the approximated solutions with and without regularization, between the exact source term and the approximated $r^\lambda$ is given,  where the stabilization functional is $\Omega({\textrm{\bf R}}) = \norm{{\textrm{\bf R}}} $. Clearly from the figure, the presence of noise in the data affects greatly the reconstruction and with the use of the Tikhonov regularization the numerical results are quite satisfactory.
	
 In Figure \ref{noises}, comparisons under different noise levels $1\%$, $2\%$ and $ 5\%$, between the exact source term  and the approximated $r^\lambda$ are given. In all three cases, the results are stable and reasonable, which are further confirmed by the errors in Table \ref{Table_2}.

 As shown in Figure \ref{noises}, the Tikhonov regularization with $\Omega({\textrm{\bf R}}) = \norm{{\textrm{\bf R}}} $ produces a stable solution, but the solution is not smooth enough. Therefore, we  minimized the cost function given in (\ref{cost_fun}) with \begin{equation}\label{smooth}
 \Omega({\textrm{\bf R}}) = \norm{ \dfrac{{\rm d} {\textrm{\bf R}}}{{\rm d} x} },
 \end{equation} 
 which produces a smooth solution (see Figures \ref{fig_1} and \ref{fig_2}). 
Better results obtained using (\ref{smooth}) is due to regularity of the source term considered. Thus, minimizing (78) with the stabilization functional given in (\ref{smooth}) will force the solution to be smooth.

In Tables 1 and 2, the relative errors of the approximated source term $r^\delta$ to the exact source $r$ are given. It can be seen that the smaller the noise level, the
better the approximative effect.% Moreover, the relative errors with different noise levels is shown in Table 2.

 \begin{figure}[http]
	\begin{center}
		\includegraphics[width=3in]{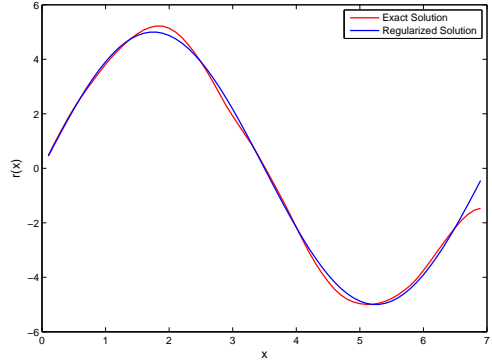}
		 \caption{Exact solution and the regularization solution with $5\%$ noisy measurement.}
		 	\label{fig_1}
	\end{center}

\end{figure}

\begin{figure}[http]
	\begin{center}
		\includegraphics[width=3in]{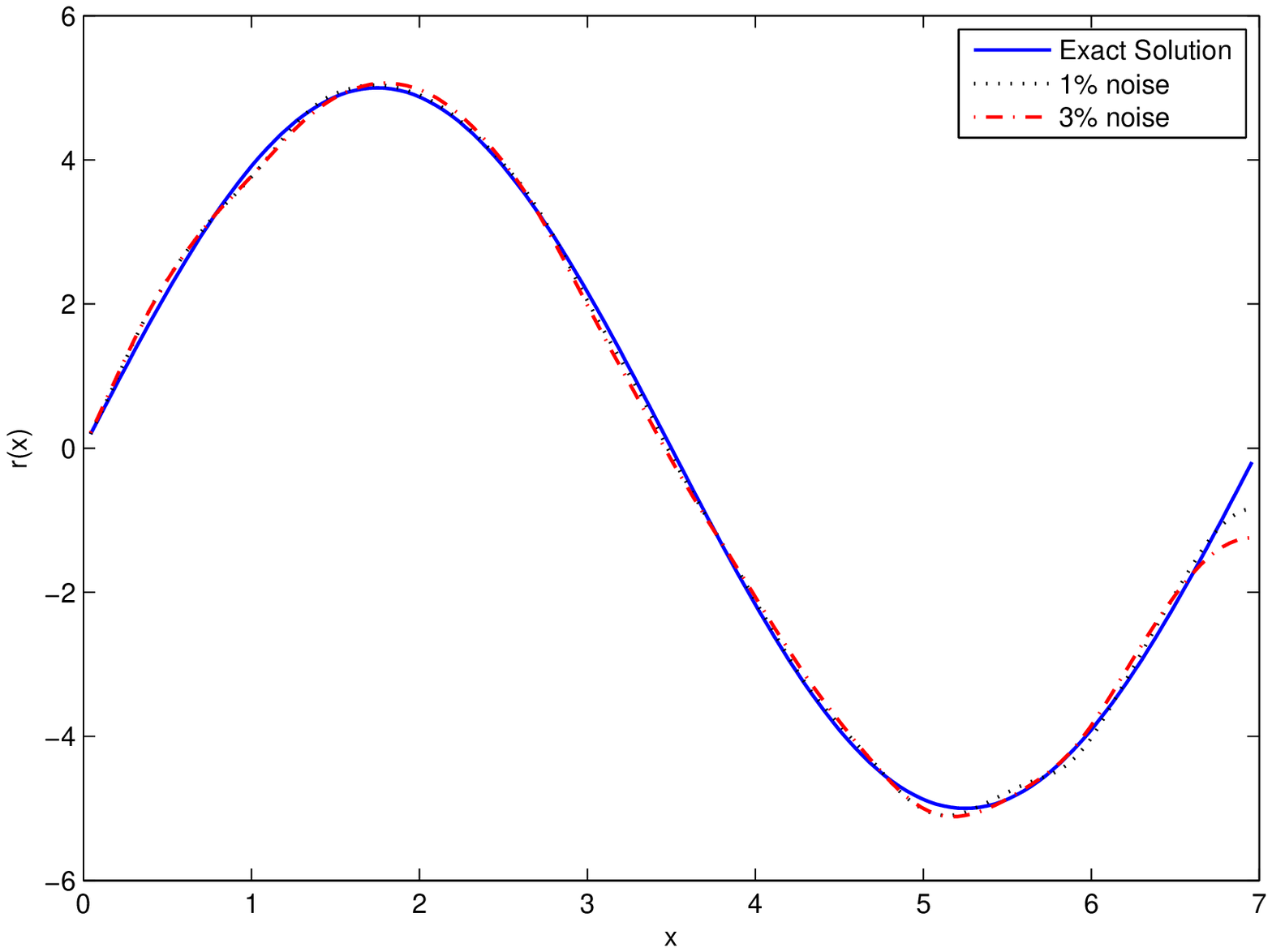}
		 \caption{The exact and the regularized solution with different noise levels.}
		 	\label{fig_2}
	\end{center}

\end{figure}

 \begin{table}[http]\label{Tab_1}
 \caption{The relative errors: $\Omega(R) = \parallel R' \parallel$, $\theta = 0.3$.}
 \begin{center}
\begin{tabular}{c c c  c}

\hline\hline
$N$ & 5\% noise& 3\% noise  & 1\% noise\\ % inserts table %heading
\hline

 $40$ & $10.18$ & $7.98$	& $4.53$\\

$60$&  $8.97$	& $6.91$ &	$4.25$\\

  $80$ & $7.81$	& $5.62$ &	$4.13$\\

   $100$ & $6.84$ &	 $5.76$  & $4.80$\\

    $120$ & $9.56$	&	$7.4$ & $4.13$ \\
    
     $140$ & $7.93$	&	$6.30$ & $4.26$ \\

    $160$ & 	$6.16$ & $4.88$  & $3.06$\\
[1ex]
\hline
\end{tabular}
\end{center}
\end{table}

 \begin{table}[http]\label{Table_2}
 \caption{The relative errors with different stabilizing functional $\Omega(R)$ when: $N=100$, $\theta=0.3$.}
 \begin{center}
\begin{tabular}{c c  c}

\hline
\hline
& $\Omega(F)$ & \\
$\sigma_{noise}$ & $\parallel F \parallel$ & $\parallel F' \parallel$ \\ % inserts table %heading

\hline

$5 \%$ & $20.23$   &  $6.84$\\
$2\%$ &  $16.06$ &  $5.76$\\
$1\%$ &  $11.14$ & $4.80$\\
[1ex]
\hline
\end{tabular}
\end{center}
\end{table}

 \section{Conclusion}
 \vspace{-0.4cm}
 In this paper, we have analyzed the mathematical properties of the direct and the inverse problem for a space fractional advection dispersion equation. For the direct problem, we have presented an analytic solution by applying the Fourier transform method and the Duhamel�s principle. Moreover, we have proved the uniqueness of the solution.  Due to the linearity of the operator, we have proved the unique determination of the source term from a final observation. Moreover,  we have analytically showed that the inverse source problem is ill-posed due to the lack stability requirements. Furthermore, a numerical solution based on the Tikhonov regularization has been presented.

\section*{Acknowledgments}
\vspace{-0.4cm}
The authors would like to express great appreciation to Prof. Manuel Ortigueira and Prof. William Rundell for their valuable suggestions and comments. 
%-----------------------------------
\bibliographystyle{plain} % Use the "unsrtnat" BibTeX style for formatting the Bibliography

\bibliography{gAPAguide} % The references (bibliography) information are stored in the file named "Bibliography.bib"

\begin{thebibliography}{10}

\bibitem{BondarenkoG}
A.~N. Bondarenko and D.~S. Ivaschenko.
\newblock Generalized sommerfeld problem for time fractional diffusion
  equation: analytical and numerical approach.
\newblock {\em Journal of Inverse Ill-posed Problems}, 17(4):321--335, June
  2009.

\bibitem{BondarenkoN}
A.~N. Bondarenko and D.~S. Ivaschenko.
\newblock Numerical methods for solving inverse problems for time fractional
  diffusion equation with variable coefficient.
\newblock {\em Journal of Inverse Ill-posed Problems}, 17(7):419--440, July
  2009.

\bibitem{BrunnerN}
Hermann Brunner, Leevan Ling, and Masahiro Yamamoto.
\newblock Numerical simulations of 2d fractional subdiffusion problems.
\newblock {\em Journal of Computational Physics}, 229(18):6613--6632, September
  2010.

\bibitem{ChenAD}
S.~Chen and F.~Liu.
\newblock Adi-euler and extrapolation methods for the two-dimensional
  fractional advection-dispersion equation.
\newblock {\em Journal of Applied Mathematics and Computing}, 26:295--311,
  2008.

\bibitem{Chi}
G.~Chi, G.~Li, and X.~Jia.
\newblock Numerical inversions of a source term in the fade with a dirichlet
  boundary condition using final observations.
\newblock {\em Computers and Mathematics with Applications}, 62(4):1619--1626,
  2011.

\bibitem{Furdui}
Ovidiu Furdui.
\newblock {\em Limits, Series, and Fractional Part Integrals: Problems in
  Mathematical Analysis}.
\newblock Springer, New York, 2013.

\bibitem{Huang}
F.~Huang and F.~Liu.
\newblock The fundamental solution of the space-time fractional
  advection-dispersion equation.
\newblock {\em Journal of Applied Mathematics and Computing}, 18(1 -
  2):339--350, 2005.

\bibitem{Jeffrey}
Alan Jeffrey.
\newblock {\em Applied Partial Differential Equations}.
\newblock Academic Press, 2002.

\bibitem{Kirsch}
Andreas Kirsch.
\newblock {\em An introduction to the mathematical theory of inverse problems}.
\newblock Springer, New York, 2011.

\bibitem{LiuN}
F.~Liu, V.~Anh, I.~Turner, and P.~Zhuang.
\newblock Numerical simulation for solute transport in fractal porous media.
\newblock {\em ANZIAM Journal}, 45:C461--C473, 2004.

\bibitem{MainardiF}
F.~Mainardi.
\newblock The fundamental solutions for the fractional diffusion-wave equation.
\newblock {\em Applied Mathematics Letters}, 9(6):23--28, November 1996.

\bibitem{MeerschaertF}
M.~Meerschaert and C.~Tadjeran.
\newblock Finite difference approximations for two-sided space-fractional
  partial differential equations.
\newblock {\em Applied Numerical Mathematics}, 56:80--90, 2006.

\bibitem{Meerschaert}
M.~M.Meerschaert and C.~Tadjeran.
\newblock Finite difference approximations for fractional advection-dispersion
  flow equations.
\newblock {\em Journal of Computational and Applied Mathematics}, 172:65--77,
  November 2004.

\bibitem{PedasN}
Arvet Pedas and Enn Tamme.
\newblock Numerical solution of nonlinear fractional differential equations by
  spline collocation methods.
\newblock {\em Computers and Mathematics with Applications}, 255:216--230,
  January 2014.

\bibitem{podl}
I.~Podlubny.
\newblock {\em Fractional differential equations}.
\newblock Academic Press, 1999.

\bibitem{QianO}
Z.~Qian.
\newblock Optimal modified method for a fractional-diffusion inverse heat
  conduction problem.
\newblock {\em Inverse Problems in Science and Engineering}, 18(4):521--533,
  2010.

\bibitem{SakamotoI}
Kenichi Sakamoto and Masahiro Yamamoto.
\newblock Initial value/boundary value problems for fractional diffusion-wave
  equations and applications to some inverse problems.
\newblock {\em Journal of Mathematical Analysis and Applications},
  382(1):426--447, October 2011.

\bibitem{SalimA}
T.~O. Salim and Ahmad el~Kahlout.
\newblock Analytical solution of time-fractional advection dispersion equation.
\newblock {\em Applications and Applied Mathematics}, 4(1):176--188, June 2009.

\bibitem{Schumer2}
R.~Schumer, D.~A. Benson, M.~M. Meerschaert, and S.~W. Wheatcraft.
\newblock Eulerian derivation of the fractional advectionÐdispersion equation.
\newblock {\em Journal of Contaminant Hydrology}, 48:69--88, March 2001.

\bibitem{Schumer1}
R.~Schumer, M.~M. Meerschaert, and B.~Baeumer.
\newblock Fractional advection-dispersion equations for modeling transport at
  the earth surface.
\newblock {\em Journal of Geophysical Research}, 114, December 2009.

\bibitem{Wei}
H.~Wei, W.~Chen, H.~Sun, and X.~Li.
\newblock A coupled method for inverse source problem of spatial fractional
  anomalous diffusion equations.
\newblock {\em Inverse Problems in Science and Engineering}, 18(7):945--956,
  2010.

\bibitem{Xiong}
X.~Xiong, Q.~Zhou, and Y.C. Hon.
\newblock An inverse problem for fractional diffusion equation in 2-dimensional
  case: Stability analysis and regularization.
\newblock {\em Journal of Mathematical Analysis and Applications},
  393:185--199, September 2012.

\bibitem{Zhang}
D.~Zhang, G.~Li, G.~Chi, X.~Jia, and H.~Li.
\newblock Numerical identification of multi-parameters in the space fractional
  advection dispersion equation by final observations.
\newblock {\em Journal of Applied Mathematics}, 2012, 2012.

\bibitem{ZhangN}
H.~Zhang, F.~Liu, and V.~Anh.
\newblock Numerical approximation of levy-feller diffusion equation and its
  probability interpretation.
\newblock {\em Journal of Computational and Applied Mathematics},
  206(2):1098--1115, 2007.

\bibitem{ZhengT}
G.~H. Zheng and T~Wei.
\newblock Two regularization methods for solving a rieszÐfeller
  space-fractional backward diffusion problem.
\newblock {\em Inverse Problems}, 26(11), October 2010.

\bibitem{Zheng}
G.H. Zheng and T.~Wei.
\newblock Spectral regularization method for cauchy problem of the time
  fractional advection dispersion equation.
\newblock {\em Journal of Computational and Applied Mathematics},
  233:2631--2640, 2010.

\end{thebibliography}

\appendix

\section*{Appendix A. Properties of the Green's Function $G_\alpha^\theta(\cdot,\cdot)$\cite{Huang}:}\label{greens_fuction}

\begin{equation}\label{equation3}
\hat G_\alpha^\theta (k,t)= e^{[i\nu k- d \psi_\theta^\alpha(k)]t }= e^{i\nu k t} e^{- d \psi_\theta^\alpha(k)t }=\hat P_1^1(k;-\nu t) \hat P^\theta_\alpha(k;td),
\end{equation}
and
\begin{eqnarray}
\begin{array}{ll}
	\hat P^\theta_\alpha(k;c) = e^{-c\psi_\theta^\alpha(k)} , & c \in \mathbb{R}.
\end{array}
\end{eqnarray}
Using the following scale rule for the Fourier transformation,
\begin{eqnarray}f(cx) \longleftrightarrow^\mathcal{F} |a|^{-1}\hat f(k/c), \end{eqnarray}
we get
\begin{eqnarray}
P_\alpha^\theta(x;c)= |c|^{-1} p_\alpha^\theta(x/c^{\frac{1}{\alpha}}),
\end{eqnarray}
which is non-negative, since $p_\alpha^\theta$ is a probability density function whose Fourier transformation is $\hat p_\alpha^\theta(k)=e^{-\psi_\theta^\alpha(k)}$.

Therefore, the inverse Fourier transformation of (\ref{equation3}) is:
\begin{eqnarray}
	G_{\alpha}^{\theta}(x,t)= \int_{- \infty}^{+ \infty} P_1^1(x-k;\nu t) P^\theta_\alpha(k;td) {\rm d}k,
\end{eqnarray}
where $G_{\alpha}^{\theta}(x,t)= \int_{- \infty}^{+ \infty} P_1^1(x-k;\nu t) P^\theta_\alpha(k;td) {\rm d}k $ is real and normalized \cite{Huang}.

\end{document}